\documentclass{amsart}
\usepackage{graphicx} 
\usepackage{xcolor}
\usepackage{amsmath, amsthm, amssymb, amsfonts}
\usepackage{hyperref}

\usepackage[skip=10pt plus1pt, indent=15pt]{parskip}

\theoremstyle{plain}
\newtheorem{theo}{Theorem}[section]

\newtheorem{lemma}[theo]{Lemma}

\newtheorem{coro}[theo]{Corollary}

\theoremstyle{definition}

\newcommand{\C}{\mathbb{C}}
\newcommand{\Z}{\mathbb{Z}}
\newcommand{\D}{\mathbb{D}}
\newcommand{\T}{\mathbb{T}}
\newcommand{\Hol}{\mathcal{H}(\D)}

\renewcommand{\Re}{\operatorname{Re}}

\title[Norm of $B$ on $H^4$]{The norm of the backward shift on $H^4$ is $\sqrt[4]{\varphi}$}

\author[K. Bampouras]{Konstantinos Bampouras}
\address{Sabanci University Tuzla Campus, Orta Mahalle, Üniversite Cadesi No:27 Tuzla, 34956 Istanbul, Turkey}
\email{kostasbaburas@gmail.com}

\author[A. Llinares]{Adri\'an Llinares}
\address{Departamento de Matem\'aticas, Facultad de Ciencias, Universidad Aut\'onoma de Madrid. C/ Francisco Tom\'as y Valiente 7, 28049 Madrid, Spain}
\email{adrian.llinares@uam.es}

\date{\today}



\begin{document}
	
	\begin{abstract}
		We prove that the backward shift operator on $H^4$ has norm equal to $\sqrt[4]{\varphi}$, with $\varphi = \frac{1 + \sqrt{5}}{2}$. Furthermore, we characterize all extremal functions; they are precisely the functions of the form
		\[
		f(z) = \mu \left( I(z) - \sqrt{\frac{1}{2\varphi}}\right),
		\]
		where $\mu \in \C$ and $I$ is an inner function with $I(0) = \sqrt{\frac{\varphi}{2}}$.
	\end{abstract}
	
	\maketitle
	
	\section{Introduction}
	
    	Let $\Hol$ be the set of all holomorphic functions in the unit disc $\D$. Given $f \in \Hol$, we define its backward shift as
    	\[
    	Bf(z) := \dfrac{f(z) - f(0)}{z}, \quad z \neq 0,
    	\]
    	and $Bf(0) := f'(0)$. If $f$ belongs to the Hardy space $H^p$ (that is, the subspace of $\Hol$ such that the quantity
    	\[
    	\| f \|_{H^p}^p := \sup_{0 \leq r < 1} \left \{ \dfrac{1}{2\pi} \int_0^{2\pi} |f(re^{it})|^p \, dt \right\}
    	\]
    	is finite), then $f$ has radial limits almost everywhere on the unit circle $\T$, this boundary function determines the norm of $f$ and, from this point of view, it is clear that
    	\[
    	\| Bf \|_{H^p} = \| f - f(0) \|_{H^p},
    	\]
    	so $B$ is a bounded operator in $H^p$. Observe that, in this setting, $B$ coincides with the adjoint operator of the (forward) shift operator $Sf (z) = zf(z)$.
    	
    	These two operators are very classical and of special relevance. Recall for instance Beurling's celebrated paper \cite{MR27954}, a capital work in Operator Theory where the invariant subspaces of $S$ are characterized (and the term ``inner function" is coined to refer to bounded analytic functions with unimodular boundary values almost everywhere). Another important result involving $B$ and inner functions is the description of its non-cyclic vectors by Douglas, Shapiro and Shields \cite{MR270196}. They showed that a function $f \in H^p$ ($1 < p <\infty$) is non-cyclic if and only if there exist two inner functions $I$ and $J$ such that
    	\[
    	\frac{f}{\overline{f}} = \frac{I}{J}, \quad \mbox{ almost everywhere on } \T.
    	\]
    	See the monographs \cite{MR1761913} and \cite{MR827223} for a detailed account of the properties of these operators. Despite the extensive list of known facts of $B$, the computation of its norm is still nowadays an open problem.
    	
    	To the best of our knowledge, there are only three values of $p$ for which the norm is known. One of course is $p=2$, which is trivial. In fact, $L^2 (\T)$ is the unique Banach function space where $B$ is contractive \cite{MR4643433}. Explicit counterexamples for contractivity can be found in \cite[Lemma 2.3]{MR4386412}. Lower bounds and other estimates (also on Bergman spaces) will be found in the forthcoming preprint \cite{JimenezVukotic}.
    	
    	If $p = \infty$, Ferguson \cite{MR3770837} showed that the norm of $B$ is equal to $2$ by testing with a sequence of inner functions with values at zero approaching the boundary.
    	
    	Ferguson also found a non-trivial estimate for $p = 1$. However, the sharp bound is due to Brevig and Seip \cite{MR4796279} who recently proved that
    	\[
    	\| Bf \|_{H^1} \leq \dfrac{2}{\sqrt{3}} \|f\|_{H^1}
    	\]
    	and that equality is attained if and only if $f$ is a constant multiple of
    	\[
    	\left(\dfrac{\sqrt{3} + I(z)}{\sqrt{3} - I(z)} \right)^2,
    	\]
    	where $I$ is an inner function vanishing at the origin.
    	
    	The main goal of this work is to broaden our understanding of the operator $B$ and its connection to inner functions. To this end, we will compute the norm of $B$ on $H^4$ and show that the extremal functions are determined by inner functions with a prescribed value at the origin. 
    	
    	Throughout this paper, we will use the notation $\varphi = \frac{1 + \sqrt{5}}{2}$. We recall that $\varphi$ satisfies the identity $\varphi^2 = \varphi + 1$, which might be useful in facilitating the understanding of some of the following arguments.
    	
    	\begin{theo} \label{thm:NormH4}
    		If $f \in H^4$, then
    		\[
    		\| B f \|_{H^4} \leq \sqrt[4]{\varphi} \| f\|_{H^4}.
    		\]
    		Moreover, equality is attained if and only if $f$ is a constant multiple of
    		\[
    		I(z) - \sqrt{\dfrac{1}{2\varphi}},
    		\]
    		where $I$ is an inner function with $I(0)=\sqrt{\frac{\varphi}{2}}$.
    	\end{theo}
	
	    Proof of Theorem \ref{thm:NormH4} is split into two parts. The upper bound for the norm is proven in Section \ref{sec:UpperBound}. It is shown in Section \ref{sec:NormAttaining} that $B$ is a norm-attaining operator on $H^4$ and all the extremal functions are described. Finally, Section \ref{sec:ConcludingRemarks} contains some concluding remarks regarding the norms of $B$ and its iterates.
	
    	\subsection*{Acknowledgments}
    	
    	The work of the second author was partially supported by the grant PID2024-160326NA-I00 from the Spanish Ministry of Science and Innovation.
	
	\section{Proof of the upper bound} \label{sec:UpperBound}
	
        To prove the upper bound, we will use the following lemma.
        
        \begin{lemma}\label{lemma}
            For every $f \in H^4$ we have that
            \begin{equation} \label{eqn:CubicBound}
                |\langle f^2, Bf \rangle|^2 \leq \dfrac{1}{2} \| f \|_{H^2}^2 ( \| f \|_{H^4}^4-\|f\|_{H^2}^4 ),
            \end{equation}
            where $\langle \cdot, \cdot \rangle$ denotes the inner product of $L^2 (\T)$. Moreover, equality holds if and only if $f$ is a constant multiple of the backward shift of an inner function.
        \end{lemma}
        
        \begin{proof}
            
            For almost every point of $\T$ we have that
            \[
            Bf(z) = \overline{z} \big( f(z) - f(0) \big).
            \]
            Since $z f^2(z)$ vanishes at the origin, we have that
            \begin{align*}
                \langle f^2, Bf \rangle & = \dfrac{1}{2\pi} \int_0^{2\pi} f^2 (e^{it}) \overline{Bf(e^{it})} \, dt \\
                & = \dfrac{1}{2\pi} \int_0^{2\pi} e^{it} f^2 (e^{it}) \big( \overline{f(e^{it})} - \overline{f(0)} \big) \, dt \\
                & = \dfrac{1}{2\pi} \int_{0}^{2\pi} e^{it} |f(e^{it})|^2 f(e^{it}) \, dt.
            \end{align*}
            If $P_+$ represents the Riesz projection (i.e., the orthogonal projection from $L^2(\T)$ onto $H^2$) we will have that the following identity holds:
            \[
            \langle f^2, B f \rangle = \langle f, P_+ ( \overline{z} |f|^2 ) \rangle.
            \]
            Hence, Cauchy-Schwarz inequality yields that
            \begin{equation} \label{eqn:CS}
                | \langle f^2, B f \rangle | \leq \| f \|_{H^2} \| P_+ ( \overline{z} |f|^2 ) \|_{H^2}.
            \end{equation}
            
            Let $\{ b_k \}_{k \in \Z}$ be the Fourier coefficients of $|f|^2 \in L^2(\T)$, which is a real-valued function and therefore we have that $b_k = \overline{b_{-k}}$ for all $k < 0$. Observe that $\| f \|_{H^2}^2 = b_0$ and
            \[
            \| P_+ (\overline{z}|f|^2) \|_{H^2}^2 = \sum_{k = 1}^\infty |b_k|^2 = \dfrac{\|f\|_{H^4}^4 - \| f \|_{H^2}^4}{2},
            \]
            which proves that inequality \eqref{eqn:CubicBound} holds.
            
            Now, assume that equality in \eqref{eqn:CubicBound} is attained for certain function $f \in H^4$. In particular, equality is also attained in \eqref{eqn:CS} and therefore there exists a constant $\lambda \in \C$ such that $P_+ (\overline{z}|f|^2) = \lambda f$.
            
            If $g(z) = z f(z)$, the identity above becomes
            \[
            \lambda g(z) = \sum_{k = 1}^\infty b_k z^k
            \]
            and then
            \[
            |g(z)|^2 = |f(z)|^2 = \|f\|_{H^2}^2 + 2 \Re \{ \lambda g(z) \}
            \]
            almost everywhere in $\T$. In other words, the function $g - \overline{\lambda} \in H^4$ has constant modulus almost everywhere and that means there exists an inner function $I$ and $\mu \in \C$ such that
            \[
            g(z) - \overline{\lambda} = \mu I(z), \quad z \in \D.
            \]
            Finally, $g$ vanishes at 0 so
            \[
            f = Bg = \mu B I.
            \]
        \end{proof}
        
        \begin{theo} \label{thm:UpperBound}
            If $f \in H^4$, then
            \[
            \| Bf \|_{H^4}^4 \leq \varphi \| f \|_{H^4}^4.
            \]
        \end{theo}
        
        \begin{proof}
            
            Let $M$ be the norm of $B$ on $H^4$. Take $\{ f_n \}_{n \geq 1}$ with $\| f_n \|_{H^4} = 1$ and
            \[
            \lim_{n \to \infty} \| Bf_n \|_{H^4} = M.
            \]
            Without loss of generality, assume that $f_n(0) \geq 0$. Also (after considering a suitable subsequence), we may assume that $\{ f_n (0) \}_{n \geq 1}$ and $\{ \| f_n \|_{H^2} \}_{n \geq 1}$ converge to two quantities $x$ and $A$, respectively. We may also assume that $\{\langle f_n^2,f_n\rangle\}_{n\geq 1}$ converges to a complex number.
            
            For every $a \in \C$, define $g(a) := \limsup_{n \to \infty} \| f_n + a \|_{H^4}^4$. Then, $g(a) \geq 1 = g(0)$ for all $a$ because of the ``extremality'' of $\{ f_n \}_{n \geq 1}$ and the fact that $B$ is invariant under translations by constants. That means that
            \[
            0 \leq g(a) - 1 \leq 4 \lim_{n \to \infty} \dfrac{1}{2\pi} \int_{0}^{2\pi} |f_n (e^{it})|^2 \Re \{{\overline{a}f_n(e^{it})} \} \, dt + O(|a|^2), \quad a \to 0
            \]
            and dividing by $|a|$ and letting $a \to 0$ we necessarily have that
            \[
            0 \leq \lim_{n \to \infty} \Re \big \{ \xi \langle f_n^2, f_n \rangle \big \}, \quad \forall \xi \in \T.
            \]
            This yields that $\lim_{n \to \infty} \langle f_n^2, f_n \rangle = 0$.
            
            Set $h_n = B f_n$ and $s = \lim_{n \to \infty} \| h_n \|_{H^2}^2 = A^2 - x^2$. Then we have that
            \begin{align}
                M^4 & = \lim_{n \to \infty} \big \| \big( f_n - f_n(0) \big)^2 \big \|_{H^2}^2 \nonumber \\
                & = \lim_{n \to \infty} \| f_n^2 - 2f_n f_n(0) + f_n^2(0)\|_{H^2}^2 \nonumber \\
                & = 1 + 4x^2 A^2 - x^4 = 1 + 4x^2 s + 3x^4. \label{eqn:aux1}
            \end{align}
            Furthermore, since
            \[
            \langle f_n^2, f_n \rangle = f_n^3(0) + 2f_n(0) \| h_n \|_{H^2}^2 + \langle z h_n^2, h_n \rangle,
            \]
            we have that $\{ \langle h_n^2, B h_n \rangle \}_{n \geq 1}$ converges to $-x^3 - 2xs$, which is negative.
            
            Applying \eqref{eqn:CubicBound} to $h_n$ and taking limits yield that
            \[
            x^2(x^2+2s)^2 \leq \dfrac{s}{2} (M^4 - s^2).
            \]
            Taking into account \eqref{eqn:aux1}, we have that $x$ and $s$ must satisfy the condition
            \begin{equation} \label{eqn:aux2}
                2x^2 (x^2 + 2s)^2 \leq s (1 + 4x^2s + 3x^4 - s^2).
            \end{equation}
            We will show from here that $M^4 \leq \varphi$.
            
            A straightforward computation shows that \eqref{eqn:aux2} is equivalent to
            \begin{equation} \label{eqn:aux3}
                2x^6 + 5x^4s+4x^2s^2+s^3 \leq s.
            \end{equation}
            Recall that $M > 1$ ($B$ is not contractive in $H^4$). According to \eqref{eqn:aux1}, this means that $x \neq 0$ and hence $s$ is different than $0$ as well. If $r = \frac{x^2}{s} > 0$, \eqref{eqn:aux3} becomes
            \[
            s^2 (r + 1)^2 (2r + 1) = s^2 (2r^3 + 5r^2 + 4r + 1) \leq 1.
            \]
            This provides an upper bound for $s^2$, which applied to \eqref{eqn:aux1} leads to
            \[
            M^4 = 1 + s^2 (4r+3r^2) \leq 1 + \dfrac{4r + 3r^2}{(r+1)^2 (2r+1)}.
            \]
            
            It remains to compute the maximum of the function
            \begin{equation} \label{eqn:AuxFunct}
                F(r)=\frac{4r+3r^2}{(r+1)^2(2r+1)}, \quad r \geq 0,
            \end{equation}
            whose derivative is
            \[
                F'(r) = -\dfrac{2(3r+2)(r^2+r-1)}{(r+1)^3(2r+1)^2}.
            \]
            It is elementary to check that the unique critical point in the range $[0, \infty)$ is where the function attains its maximum. That critical point is $r = \frac{\sqrt{5}-1}{2}$ which is, in fact, a fixed point of $F$. This completes the proof.
        \end{proof}
	
	\section{Lower bound and characterization of extremal functions} \label{sec:NormAttaining}
	
    	In this section we will deduce that the upper bound in Theorem \ref{thm:UpperBound} is sharp. Actually, we will find all of the extremal functions. That is, all of the functions $f \in H^4$ such that
    	\[
    	   \| B f \|_{H^4}^4 = \varphi \| f \|_{H^4}^4.
    	\]
    	
    	\begin{lemma} \label{lem:NormAttaining}
    		The backward shift is a norm-attaining operator in $H^4$.
    	\end{lemma}
    	
    	\begin{proof}
    		As a consequence of Theorem \ref{thm:UpperBound}, it is enough to find a function $f$ with $\| Bf \|_{H^4}^4 = \varphi \| f \|_{H^4}^4$.
    		
    		Let $I$ be any inner function with $I(0) = \sqrt{\frac{\varphi}{2}}$ and set
        		\[
        		      f(z) = I (z) - \sqrt{\dfrac{1}{2\varphi}}.
        		\]
        		On the one hand, $Bf$ coincides with $B I$ and for every inner function the identity
        		\begin{equation} \label{eqn:NormInner}
        			\| B I \|_{H^4}^4 = 1 - |I(0)|^4
        		\end{equation}
        		holds, so $\| Bf \|_{H^4}^4 = 1 - \frac{\varphi^2}{4}$.
    		
    		On the other hand, since $\varphi^2 = \varphi + 1$ we have that
        		\begin{align*}
        			\| f \|_{H^4}^4 & = \left \| I^2 - \sqrt{\dfrac{2}{\varphi}} I + \dfrac{1}{2\varphi} \right \|_{H^2}^2 \\
            			&  = 1 + \dfrac{2}{\varphi} + \dfrac{1}{4 \varphi^2} - 2 + \dfrac{1}{2} - \dfrac{1}{\varphi} \\
            			& = \dfrac{1}{\varphi} -\dfrac{1}{2} + \dfrac{1}{4 \varphi^2} \\
            			& = \dfrac{1}{\varphi} \left(1 - \dfrac{\varphi}{2} + \dfrac{\varphi - 1}{4} \right) = \dfrac{1}{\varphi} \left( 1 - \dfrac{\varphi^2}{4}\right).
        		\end{align*}
    		
    		Thus, the function $f$ must be an extremal function.
    	\end{proof}
    	
    	\begin{coro}
    		A function $f \in H^4$ is extremal for the backward shift if and only if there exists $\mu \in \C$ and inner function $I$ with $I(0) = \sqrt{\frac{\varphi}{2}}$ such that
    		\[
    		      f(z) = \mu \left( I(z) - \sqrt{\dfrac{1}{2\varphi}}\right).
    		\]
    	\end{coro}
    	
    	\begin{proof}
    		Let $f$ be a normalized extremal function with $f(0) > 0$. Then we can repeat the proof of Theorem \ref{thm:UpperBound} choosing $f_n = f$ for all $n$.
    		In particular, the function $h = B f$ must attain the equality in \eqref{eqn:CubicBound}. Hence, there exist $\mu \in \C \setminus \{ 0 \}$ and an inner function $I$ with $I(0) > 0$ such that
        		\[
        		      Bf(z) = \mu BI (z),
        		\]
        	    so $|\mu|^4 \big(1 - I^4(0) \big) = \varphi$ by the identity \eqref{eqn:NormInner}. The same argument yields that $\langle  B^2f,(Bf)^2\rangle$ is negative and
        		\[
        		-\langle  B^2f,(Bf)^2\rangle=|\mu|^2\overline{\mu}I(0)\big (1-I^2(0) \big)
        		\]
        		and thus $\mu$ is real and positive. Also, by the proof of Theorem \ref{thm:UpperBound}, we have that the extremal case must satisfy 
        		\[
                    \| Bf \|_{H^2}^4 (r + 1)^2 (2r + 1) = 1, \quad \mbox{with } r=\frac{\sqrt{5}-1}{2} = \varphi - 1,
                \]
                and thus
        		\[
        		\mu^4 \big(1 - I^2(0) \big)^2 =\| Bf \|_{H^2}^4 = \dfrac{1}{\varphi^2(2 \varphi - 1)} = \dfrac{1}{3\varphi + 1}
        		\]
        		and hence $I^2(0) = \frac{\varphi}{2}$ and $\mu^4 = \frac{4\varphi}{3 - \varphi}$.  Maximizing the function \eqref{eqn:AuxFunct} requires that
        		\[
        		f^2(0) = (\varphi - 1) \| Bf \|_{H^2}^2 = \dfrac{\varphi - 1}{\sqrt{3 \varphi + 1}}.
        		\]
        		This means that
        		\[
        		f(z) = \mu \left( I(z) - \sqrt{\dfrac{\varphi}{2}}\right) + f(0) = \mu \left( I(z) - \sqrt{\dfrac{1}{2\varphi}}\right).
        		\]
    		
    		The other implication follows from the proof of Lemma \ref{lem:NormAttaining}.	
    	\end{proof}

	\section{Concluding remarks} \label{sec:ConcludingRemarks}
	
    	The first remark of this section is regarding the essential norm of $B$. Recall that the essential norm of a bounded operator $T$ on $H^p$ is defined as the infimum of all the norms of compact perturbations of $T$. Shargorodsky \cite[Theorem~5.1]{MR4169308} showed that the norm of the backward shift on $H^p$, $1 < p < \infty$, coincides with its essential norm. Thus, combining this result with Theorem~\ref{thm:NormH4} we get the next corollary.
    	
    	\begin{coro}
    		The essential norm of $B$ on $H^4$ is equal to $\sqrt[4]{\varphi}$.
    	\end{coro}
    	
    	The final remark is about the iterates of $B$ and their connection to the Riesz projection. For every $n \geq 2$ let $B_n := B \circ B_{n-1}$ and $B_1$ is just the backward shift. Since $B_n = B_{n+1} \circ S$, it is then clear that the sequence of their norms, $\{ \| B_n \|_{p} \}_{n \geq 1}$, is increasing. Furthermore, trigonometric polynomials are dense in $L^p ( \T)$ and hence it can be checked that
    	\begin{equation} \label{eqn:Riesz}
    		\lim_{n \to \infty} \| B_n \|_{p} = \| P_{+} \|_{p} = \dfrac{1}{\sin \frac{\pi}{p}}, \quad 1 < p < \infty.
    	\end{equation}
    	Thus, we have the following:
    	
    	\begin{coro}
    		For every $n \geq 2$ we have that
    		\[
    		\sqrt[4]{\varphi} \leq \| B_n \|_{H^4} \leq \sqrt{2}.
    		\]
    	\end{coro}
    	
    	An interesting (and challenging) problem would be to establish the veracity of \eqref{eqn:Riesz} without having to rely on the fact that the norm of the Riesz projection is currently known. If this were possible, it would constitute an alternative proof from the original by Hollenbeck and Verbitsky \cite{MR1780482} without the need to use plurisubharmonic functions. Although these functions are a very powerful tool for solving extremal problems (in addition to \cite{MR1780482}, see for instance \cite{MR4009386} or \cite{MR4721797}), a constructive argument based of the iterates of $B$ would help to improve our understanding of the behavior of $P_+$.

	\bibliographystyle{amsplain}
	\bibliography{Arxiv_biblio}

\providecommand{\bysame}{\leavevmode\hbox to3em{\hrulefill}\thinspace}
\providecommand{\MR}{\relax\ifhmode\unskip\space\fi MR }
\providecommand{\MRhref}[2]{%
  \href{http://www.ams.org/mathscinet-getitem?mr=#1}{#2}
}
\providecommand{\href}[2]{#2}
\begin{thebibliography}{10}

\bibitem{MR27954}
A.~Beurling, \emph{On two problems concerning linear transformations in
  {H}ilbert space}, Acta Math. \textbf{81} (1948), 239--255. \MR{27954}

\bibitem{MR4386412}
O.~F. Brevig, J.~Ortega-Cerd\`a, and K.~Seip, \emph{Idempotent {F}ourier
  multipliers acting contractively on {$H^p$} spaces}, Geom. Funct. Anal.
  \textbf{31} (2021), no.~6, 1377--1413. \MR{4386412}

\bibitem{MR4796279}
O.~F. Brevig and K.~Seip, \emph{The norm of the backward shift on {$H^1$} is
  {$\frac{2}{\sqrt 3}$}}, Pure Appl. Funct. Anal. \textbf{9} (2024), no.~4,
  991--994. \MR{4796279}

\bibitem{MR1761913}
J.~A. Cima and W.~T. Ross, \emph{The {B}ackward {S}hift on the {H}ardy
  {S}pace}, Mathematical Surveys and Monographs, vol.~79, American Mathematical
  Society, Providence, RI, 2000. \MR{1761913}

\bibitem{MR270196}
R.~G. Douglas, H.~S. Shapiro, and A.~L. Shields, \emph{Cyclic vectors and
  invariant subspaces for the backward shift operator}, Ann. Inst. Fourier
  (Grenoble) \textbf{20} (1970), 37--76. \MR{270196}

\bibitem{MR3770837}
T.~Ferguson, \emph{Bounds on the norm of the backward shift and related
  operators in {H}ardy and {B}ergman spaces}, Illinois J. Math. \textbf{61}
  (2017), no.~1-2, 81--96. \MR{3770837}

\bibitem{MR1780482}
B.~Hollenbeck and I.~E. Verbitsky, \emph{Best constants for the {R}iesz
  projection}, J. Funct. Anal. \textbf{175} (2000), no.~2, 370--392.
  \MR{1780482}

\bibitem{JimenezVukotic}
I.~Jim\'enez and D.~Vukoti\'c, \emph{Norm estimates for the backward shift}, To
  appear in arXiv, 2026.

\bibitem{MR4009386}
D.~Kalaj, \emph{On {R}iesz type inequalities for harmonic mappings on the unit
  disk}, Trans. Amer. Math. Soc. \textbf{372} (2019), no.~6, 4031--4051.
  \MR{4009386}

\bibitem{MR4643433}
O.~Karlovych and E.~Shargorodsky, \emph{When are the norms of the {R}iesz
  projection and the backward shift operator equal to one?}, J. Funct. Anal.
  \textbf{285} (2023), no.~12, Paper No. 110158, 29. \MR{4643433}

\bibitem{MR4721797}
P.~Melentijevi\'c, \emph{Hollenbeck-{V}erbitsky conjecture on best constant
  inequalities for analytic and co-analytic projections}, Math. Ann.
  \textbf{388} (2024), no.~4, 4405--4448. \MR{4721797}

\bibitem{MR827223}
N.~K. Nikolskii, \emph{Treatise on the {S}hift {O}perator}, Grundlehren der
  mathematischen Wissenschaften [Fundamental Principles of Mathematical
  Sciences], vol. 273, Springer-Verlag, Berlin, 1986, Spectral function theory,
  With an appendix by S. V. Hru\v s\v cev [S. V. Khrushch\"ev] and V. V.
  Peller, Translated from the Russian by Jaak Peetre. \MR{827223}

\bibitem{MR4169308}
E.~Shargorodsky, \emph{On the essential norms of {T}oeplitz operators with
  continuous symbols}, J. Funct. Anal. \textbf{280} (2021), no.~2, Paper No.
  108835, 11. \MR{4169308}

\end{thebibliography}
	
\end{document}